\documentclass{article}
\usepackage{amsmath,amsthm,amssymb,amscd,pict2e}

\newcommand{\ga}{\alpha}
\newcommand{\gb}{\beta}

\newcommand{\gd}{\delta}
\newcommand{\gw}{\omega}

\newcommand{\gS}{\Sigma}
\newcommand{\gs}{\sigma}
\newcommand{\eps}{\varepsilon}

\newcommand{\dotggen}{\dot\gamma_{\mathrm{gen}}}

\newcommand{\coll}{\mathrm{Coll}}

\newcommand{\cantor}{2^\gw}

\newcommand{\dotxgen}{{\dot x}_{\mathit{gen}}}

\newcommand{\supp}{\mathrm{supp}}

\newcommand{\dom}{\mathrm{dom}}

\newcommand{\power}{\mathcal{P}}

\newtheorem{theorem}{Theorem}[section]

\newtheorem{claim}[theorem]{Claim}
\newtheorem{corollary}[theorem]{Corollary}
\newtheorem{fact}[theorem]{Fact}
\newtheorem{proposition}[theorem]{Proposition}

\theoremstyle{definition}
\newtheorem{definition}[theorem]{Definition}
\newtheorem{example}[theorem]{Example}

\title{Noetherian spaces in choiceless set theory\footnote{2020 AMS subject classification 03E25, 03E35.}}

\author{
Jind{\v r}ich Zapletal\\
University of Florida\\
Institute of Mathematics, Czech Academy of Sciences}

\begin{document}
\maketitle

\begin{abstract}
I prove several independence results in the choiceless ZF+DC theory which separate algebraic and non-algebraic consequences of the axiom of choice.
\end{abstract}

\section{Introduction}
Geometric set theory \cite{z:geometric} was developed in part to produce consistency results in choiceless ZF+DC set theory regarding various $\gS^2_1$ sentences. In this paper, I produce several such consistency results which separate $\gS^2_1$ sentences dealing with algebraic topics from those which do not. While the distinction may seem vague, the techniques of the paper show that there is in fact a clearly visible fracture line. 

I isolate the notion of Noetherian balanced forcing and show that in the generic extensions of the choiceless Solovay model by such forcings, numerous $\gS^2_1$ consequences of the axiom of choice fail. The following theorem, stated using the conventions of geometric set theory \cite[Convention 1.7.18]{z:geometric}, provides a somewhat representative list.

\begin{theorem}
\label{firsttheorem}
In cofinally Noetherian balanced extensions of the choiceless Solovay model,

\begin{enumerate}
\item if $A_0, A_1\subset 3^\gw$ are non-meager sets, then there are points $x_0\in A_0$ and $x_1\in A_1$ such that the set $\{n\in\gw\colon x_0(n)=x_1(n)\}$ is finite;
\item if $\Gamma\curvearrowright X$ is a turbulent Polish group action with all orbits dense and meager, and $A_0, A_1\subset X$ are non-meager sets then there are points $x_0\in A_0$ and $x_1\in A_1$ which are orbit equivalent;
\item (if the extension is $\gs$-closed) outer Lebesgue measure is continuous in arbitrary increasing unions.
\end{enumerate}
\end{theorem}

\noindent The theorem implies in particular that there is no non-principal ultrafilter on natural numbers in the extensions. Such an ultrafilter $U$ could be used to divide sequences in $3^\gw$ into three non-meager sets according to their $U$-prevalent value, violating item (1). It also implies that if $\Gamma\curvearrowright X$ is a turbulent Polish group action with an orbit equivalence relation $E$, then every $E$-invariant function $f$ from $X$ to another Polish space $Y$ stabilizes on a co-meager set in the extensions. The reason is that for each basic open set $O\subset Y$, the set $\{x\in X\colon f(x)\in O\}$ is $E$-invariant, and therefore meager or co-meager by (2).

It turns out that many posets adding objects to the choiceless Solovay model which are related to algebraic applications of the Axiom of Choice are Noetherian balanced. The following theorem provides a good list.

\begin{theorem}
\label{secondtheorem}
The following Suslin posets are $\gs$-closed and Noetherian balanced:

\begin{enumerate}
\item a poset adding a Hamel basis, for a Polish field $X$ over a countable subfield $F$ \cite[Example 6.3.6]{z:geometric};
\item a poset adding a transcendence basis, for a Polish field $X$ over a countable subfield $F$ \cite[Example 6.3.10]{z:geometric} ;
\item a poset coloring a given closed Noetherian graph without an uncountable clique \cite{z:ngraphs};
\item a poset coloring $\gs$-algebraic graphs without an uncountable clique \cite{z:distance};
\item a poset coloring the hypergraph of isosceles triangles in $\mathbb{R}^2$ \cite{zhou:isosceles};
\item a poset coloring a given redundant $\gs$-algebraic hypergraph \cite{z:redundant}.
\end{enumerate}
\end{theorem}

\noindent There are attendant consistency results too numerous to list separately. For example, it is consistent relative to an inaccessible cardinal that ZF+DC holds, there is a transcendence basis for $\mathbb{R}$ over $\mathbb{Q}$, there is no nonprincipal ultrafilter on $\gw$, and outer Lebesgue measure is continuous under arbitrary increasing unions. Since the class of Noetherian balanced posets is closed under countable product, the consistency results can be even combined.

Section~\ref{nextensionsection} discusses the notion of a subbasis for a Noetherian topology and that of a mutually Noetherian pair of generic extensions of a model of ZFC. This is an instrumental weakening of mutual genericity. In Section~\ref{example1section} I produce several interesting mutually Noetherian pairs of generic extensions, notably one induced by a turbulent action of a Polish group. Section~\ref{preservationsection} defines the notion of a Noetherian balanced Suslin forcing--this is a forcing in which conditions can be successfully amalgamated across mutually Noetherian pairs of generic extensions. There are several attendant preservation theorems for Noetherian balanced Suslin forcings, resulting in the proof of various items of Theorem~\ref{firsttheorem}. Section~\ref{example2section} lists some Noetherian balanced Suslin forcings, proving Theorem~\ref{secondtheorem}. Section~\ref{example3section} finally shows that the conclusions of Theorem~\ref{firsttheorem} can be violated in suitable balanced, not Noetherian balanced, extensions of the Solovay model. 

The paper uses the notation standard of \cite{jech:set}. In matters pertaining to geometric set theory it follows the terminology and notation of \cite{z:geometric}. DC denotes the Axiom of Dependent Choices. The inaccessible cardinal in the assumptions of Theorems~\ref{firsttheorem} and~\ref{secondtheorem} is used to start the method of balanced forcing as in \cite{z:geometric}; I do not know if it is necessary. 

Thanks are due to Chris Lambie-Hanson for suggesting the terminology of Definition~\ref{subbasisdefinition}. The author was partially supported by grant EXPRO 20-31529X of GA \v CR.

\section{Noetherian pairs of generic extensions}
\label{nextensionsection}

The technology of this paper rests on the treatment of Noetherian topologies from a descriptive set theoretic point of view. After plenty of hesitation, I have decided in the favor of the following presentation.

\begin{definition}
\label{subbasisdefinition}
Let $X, Y$ be Polish spaces and $C\subset Y\times X$ be a closed set. The set $C$ is a \emph{Noetherian subbasis} if there is no sequence $\langle b_n\colon n\in\gw\rangle$ of finite subsets of $Y$ such that the sets $D_n=\bigcap_{y\in b_n}C_y$ for $n\in\gw$ strictly decrease with respect to inclusion.
\end{definition}

\noindent It is not difficult to see that being a Noetherian subbasis is a coanalytic property of the set $C$, therefore transitive among well-founded models of set theory. The terminology is justified by the following routine proposition.

\begin{proposition}
Let $C\subset Y\times X$ be a Noetherian subbasis. Then the smallest topology on $X$ containing all complements of vertical sections of $C$ is Noetherian.
\end{proposition}

\begin{proof}
Close the collection of vertical sections of $C$ first under finite intersections, then under finite unions, obtaining a collection $\mathcal{A}$. 

\begin{claim}
\label{claim1}
$\mathcal{A}$ contains no infinite sequences strictly decreasing with respect to inclusion.
\end{claim}

\begin{proof}
Suppose towards a contradiction that $\langle A_n\colon n\in\gw\rangle$ is such a sequence. The definitory property of a Noetherian subbasis shows that there is a finite set $b\subset Y$ such that, writing $C_b=\bigcap_{y\in b}C_y$, the sequence of intersections $\langle A_n\cap C_b\colon n\in\gw\rangle$ does not stabilize and the set $C_b$ is inclusion-minimal such. There must be $n\in\gw$ such that $A_n\cap C_b$ is a proper subset of $C_b$. Find a finite collection $\{c_i\colon i\in m\}$ of finite subsets of $Y$ such that $A_n=\bigcup_i B_i$ where $B_i=\bigcap_{y\in c_i}C_y$. It is clear that there must be $i\in m$ such that the sequence of intersections $\langle A_k\cap C_b\cap B_i\colon k\in\gw\rangle$ does not stabilize, contradicting the minimal choice of the set $C_b$.
\end{proof}

\begin{claim}
\label{claim2}
$\mathcal{A}$ is closed under finite unions and arbitrary intersections.
\end{claim}

\begin{proof}
$\mathcal{A}$ is closed under finite unions and finite intersections by its definition. Now, if $a\subset\mathcal{A}$ is any set, then by the previous claim there is a finite subset $b\subset a$ such that $\bigcap a=\bigcap b$. Since $\bigcap b\in\mathcal{A}$, it is also the case that $\bigcap a\in\mathcal{A}$ as desired.
\end{proof}

\noindent It is now clear from Claim~\ref{claim2} that the collection of the complements of sets in $\mathcal{A}$ is a topology, it is the smallest topology containing all complements of vertical sections of the set $C$, and that it is Noetherian by Claim~\ref{claim1}. 
\end{proof} 

\noindent The following observation will be key in numerous contexts.

\begin{proposition}
\label{basicproposition}
Let $X, Y$ be Polish spaces and $C\subset Y\times X$ be a Noetherian subbasis. If $a\subset Y$ is any set, then there is a finite set $b\subset a$ such that $\bigcap_{y\in a}C_y=\bigcap_{y\in b}C_y$.
\end{proposition}

\begin{proof}
If this failed for some set $a\subset Y$, it would be possible to build finite sets $b_n\subset a$ for $n\in\gw$ such that $b_n\subset b_{n+1}$ and $\bigcap_{y\in b_n}C_y\neq \bigcap_{y\in b_{n+1}}C_y$. The sets $b_n$ for $n\in\gw$ would then violate the Noetherian property of the set $C$.
\end{proof}

\noindent In the way of an illustration, the following example of a Noetherian subbasis is ubiquitous in algebra.

\begin{example}
Let $n\in\gw$ be a number and let $X$ be a Polish field.  Let $Y$ be a Polish space indexing all polynomials with $n$ free variables and coefficients in $X$. Then $C\subset Y\times X^n$, defined by $\langle y, \bar x\rangle\in C$ if and only if $y(\bar x)=0$ is a Noetherian subbasis. The nonexistence of decreasing sequences of finite intersections of vertical sections of the set $C$ is verified via the Hilbert Basis Theorem \cite{hilbert:basis}.
\end{example}

\noindent The following definition is the key tool for connecting Noetherian topologies with geometric set theory.

\begin{definition}
\label{extdefinition}
Let $X, Y$ be Polish spaces and $C\subset Y\times X$ be a Noetherian subbasis. Let $M$ be a transitive model of set theory coding $X, Y$, and $C$. Let $A\subset X$ be a set. The symbol $C(M, A)$ denotes the set $\bigcap\{C_y\colon y\in Y\cap M$ and $A\subseteq C_y\}$. If the set in the scope of the $\bigcap$ sign is empty, then set $C(M, A)=X$.
\end{definition}

\noindent The most common situation for applying this definition is that neither the set $A$ nor any of its elements belong to the model $M$. Note that the set $C(M, A)$ is coded in the model $M$ as the intersection defining it can be obtained as an intersection $\bigcap_{y\in b}C_y$ for a suitable finite set $b\subset Y\cap M$ by Proposition~\ref{basicproposition}. 

\begin{definition}
Let $V[G_0]$ and $V[G_1]$ be generic extensions of $V$ inside an ambient generic extension.  Say that $V[G_1]$ is a \emph{Noetherian to} $V[G_0]$ (over $V$) if for every Noetherian subbasis $C\subset Y\times X$ coded in $V$ and for every set $A\subset X$ in  $V[G_1]$, $C(V[G_0], A)=C(V, A)$. Say that the extensions $V[G_0], V[G_1]$ are \emph{mutually Noetherian} if each is Noetherian to the other. Say that $V[G_1]$ is \emph{$K_\gs$-Noetherian to} $V[G_0]$ if the above equation holds for all $K_\gs$ Noetherian subbases coded in $V$, and similarly for the mutual $K_\gs$-Noetherian relation.
\end{definition}

\noindent It is useful to introduce special terminology for $K_\gs$ Noetherian subbases, since most natural examples are $K_\gs$ and $K_\gs$ subbases are easier to deal with in several respects which appear elsewhere \cite{z:distance}. Similar notions of perpendicularity always have a friendly relationship with product forcing, as recorded in the following proposition.

\begin{proposition}
\label{productproposition}
Let $n\geq 1$ be a number. Let $V[G_0], V[G_1]$ be generic extensions and $V[G_1]$ is Noetherian to $V[G_0]$. Suppose that $P_0\in V[G_0]$ and $P_1\in V[G_1]$ be posets and $H_0\subset P_0$ and $H_1\subset P_1$ be filters mutually generic over $V[G_0, G_1]$. Then $V[G_1][H_1]$ is Noetherian to $V[G_0][H_0]$. Similarly for $K_\gs$-Noetherian.
\end{proposition}

\begin{proof}
Work in the model $V[G_0, G_1]$ and consider the poset $P_0\times P_1$. Let $X, Y$ be Polish spaces and $C\subset Y\times X$ be a Noetherian subbasis coded in $V$. Let $p_0\in P_0$ and $p_1\in P_1$ be conditions and let $\tau_0, \tau_1$ be respective $P_0, P_1$-names in the models $V[G_0], V[G_1]$ such that $p_0\Vdash\tau_0\subset X$, $\tau_1\subset Y$ is a finite set, and $\langle p_0, p_1\rangle\Vdash\bigcap_{y\in\tau_0}C_y=C(V[G_0][H_0], \tau_1)$; I must produce a ground model coded closed set $D\subset X$ such that $p_0\Vdash\bigcap_{y\in\tau_0}C_y=D$. 

Working in $V[G_1]$, form the closed set $A\subset X$ as $A=X\setminus \bigcup\{O\colon O\subset X$ is open and $p_1\Vdash O\cap\tau_1=0$. By the initial assumptions on the models $V[G_0]$ and $V[G_1]$, $C(V[G_0], A)=C(V, A)$ holds; write $D$ for the common value. Observe that $p_1\Vdash\tau_1\subset D$.  It will be enough to show that $p_0\Vdash \bigcap_{y\in\tau_0}C_y=D$.

Since $p_1\Vdash\tau_1\subset D$, the only way how the equality can fail is that there is a condition $p'_0\leq p_0$ forcing $\tau_0$ to be a proper subset of $D$. Working in $V[G_0]$ under this assumption, let $M_0$ be a countable elementary submodel of some large structure containing $\tau_0, D$, and $p'_0$. Let $h_0\subset P_0\cap M_0$ be a filter generic over the model $M_0$ and let $b=\tau_0/h_0$. The set $E=\bigcap_{y\in b}C_y$ is a proper subset of $D$, so $A\subseteq E$ fails. Thus, there must be a basic open set $O\subset X$ disjoint from $E$ which contains some element of the set $A$. By the definitions, this means that there is a condition $p''_0\leq p'_0$ in the filter $h_0$ which forces $\tau_0\cap O=0$, and a condition $p'_1\leq p_1$ which forces $\tau_1\cap O\neq 0$. This contradicts the initial assumptions on the conditions $p_0, p_1$.
\end{proof}

\begin{corollary}
Mutually generic extensions are mutually Noetherian.
\end{corollary}

\noindent The following more elaborate product forcing feature will come handy as well.

\begin{proposition}
\label{overlapproposition}
Suppose that $V[G_0, G_1]$ is a generic extension. Suppose that $V[G_0, G_1][K_0]$ and $V[G_0, G_1][K_1]$ are mutually generic extensions of $V[G_0, G_1]$ and $H_0\in V[G_0, G_1][K_0]$ and $H_1\in V[G_0, G_1][K_1]$ are filters such that 

\begin{enumerate}
\item $G_0\in V[H_0]$ and $G_1\in V[H_1]$;
\item $V[G_0, G_1]$ is Noetherian to $V[H_0]$ over $V[G_0]$;
\item $V[G_1]$ is Noetherian to $V[H_0]$ over $V$.
\end{enumerate}

\noindent Then $V[H_1]$ is Noetherian to $V[H_0]$ over $V$. Similarly for $K_\gs$-Noetherian.
\end{proposition}

\begin{proof}
Let $X, Y$ be Polish spaces, $C\subset Y\times X$ a Noetherian basis coded in $V$, and suppose that $A\subset X$ is a set in $V[H_1]$. To see that $C(V[H_0], A)=C(V, A)$, consider the set $B=C(V[G_0, G_1], A)$. By the mutual genericity of $K_0, K_1$ and Proposition~\ref{productproposition}, $C(V[G_0, G_1, K_0], A)=B$. Now, observe that equality $C(V[H_0], A)=C(V[H_0], B)$ must hold: the left-to-right inclusion follows from $A\subseteq B$, and the right-to-left inclusion follows the fact that $B=C(V[G_0, G_1, K_0], A)\subseteq C(V[H_0], A)$. By (2), $C(V[H_0], B)=C(V[G_0, B])$ holds. Observe that $C(V[G_0], A)=C(V[G_0], B)$ must hold: the left-to-right inclusion follows from $A\subseteq B$ again, and the right-to-left follows from the fact that $B=C(V[G_0, G_1], A)\subseteq C[V[G_0], A)$. By (3), $C(V[G_0], A)=C(V, A)$ holds. In conclusion, $C(V[H_0], A)=C(V, A)$ holds as required. 
\end{proof}

\noindent The following proposition collects several useful features of mutually Noetherian extensions.

\begin{proposition}
\label{ahproposition}
Let $V[G_0], V[G_1]$ be mutually Noetherian extensions. Then

\begin{enumerate}
\item $\cantor\cap V[G_0]\cap V[G_1]=\cantor\cap V$;
\item let $X$ be a Polish field, $p(\bar v_0, \bar v_1)$ a polynomial with all parameters in $V$ and all free variables listed. Let $\bar x_0\in V[G_0]$ and $\bar x_1\in V[G_1]$ be tuples such that $X\models p(\bar x_0, \bar x_1)=0$. Then there is a tuple $\bar x_1'\in V$ such that $X\models p(\bar x_0, \bar x'_1)=0$;
\item let $\phi(\bar v_0, \bar v_1)$ be a formula of the language of real closed fields with real parameters in $V$, with all free variables listed. Let $V[G_0], V[G_1]$ be mutually Noetherian extensions and $\bar x_0\in V[G_0]$ and $\bar x_1\in V[G_1]$ be tuples of reals such that $\mathbb{R}\models \phi(\bar x_0, \bar x_1)$ holds. Then there is a tuple $\bar x'_1\in V$ of reals such that $\mathbb{R}\models\phi(\bar x_0, \bar x'_1)$ holds.
\end{enumerate}
\end{proposition}

\begin{proof}
For (1), let $X=Y=\cantor$ and let $C\subset Y\times X$ be the diagonal; this is a Noetherian basis. Now, suppose that $x\in X\cap V[G_0]\cap V[G_1]$ is a point. By the Noetherian assumption,  $C(V[G_0], \{x\})=C(V, \{x\})=\{x\}$. Since the set $C(V, \{x\})$ is coded in $V$, this is possible only if $x\in V$.

For (2), let $n_0=|\bar v_0|$ and $n_1=|\bar v_1|$. Consider the closed set $C\subset X^{n_0}\times X^{n_1}$ consisting exactly the pairs $\langle \bar y_0, \bar y_1\rangle$ if $p(\bar y_0, \bar y_1)=0$; this is a Noetherian subbasis by the Hilbert Basis Theorem. The Noetherian assumption shows that $C(V[G_0], \{\bar x_1\})=C(V, \{\bar x_1\})$. Any tuple $x'_1$ in this set in the ground model will work as in (2), since $C(V[G_0], \bar x_1)\subseteq\{\bar y_1\colon p(\bar x_0, \bar y_1)=0\}$.

For (3), let $n_0=|\bar v_0|$ and $n_1=|\bar v_1|$. Use the quantifier elimination theorem for real closed fields \cite[Theorem 3.3.15]{marker:book} to assume that $\phi$ is quantifier-free. Then $\phi$ is a boolean combination of statements of the form $p(\bar v_0, \bar v_1)>0$ and $p(\bar v_0, \bar v_1)=0$ for some polynomials $p$ with coefficients in $V$. Let $p_i$ for $i\in j$ be a list of all polynomials used in this boolean combination. Let $a\subset j$ be the set of all indices $i$ such that $p_i(\bar x_0, \bar x_1)=0$ and let $q(\bar v_0, \bar v_1)=\gS_{i\in a}p_i(\bar v_0, \bar v_1)^2$. Consider the set $C\subset\mathbb{R}^{n_0}\times\mathbb{R}^{n_1}$ consisting of all pairs $\langle\bar y_0, \bar y_1\rangle$ such that $q(\bar y_0, \bar y_1)=0$. $C$ is a Noetherian subbasis by the Hilbert Basis Theorem, and by the Noetherian assumption on the models $V[G_0], V[G_1]$, $C(V[G_0],\bar x_1)=C(V, \bar x_1)$. Let $O\subset \mathbb{R}^{n_1}$ be a rational open box containing $\bar x_0$ such that for each $i\in j\setminus a$, the values $p_i(\bar x_0, \bar y_1)$ have the same sign for all $\bar y_1\in O$. By a Mostowski absoluteness argument, there must be a point $\bar x'_1\in O\cap C(V, \bar x_1)$ in $V$ since there is such  a point $\bar x_1$ in $V[G_1]$. The point $\bar x'_1$ works as required. 
\end{proof}

\noindent It may seem difficult to verify that given two generic extensions are mutually Noetherian. In this paper, this is always done using the following \emph{duplication criterion}.

\begin{definition}
\label{duplicationdefinition}
Suppose that $P$ is a poset, $\tau_0, \tau_1$ are $P$-names for subsets of the ground model. Say that $\tau_0$ is \emph{duplicable over} $\tau_1$ in $P$ if for every condition $p\in P$ there is a generic extension $V[K]$ and in it, a sequence $\langle H_\ga\colon\ga\in\kappa\rangle$ such that

\begin{enumerate}
\item $\kappa$ is an uncountable ordinal in $V[K]$;
\item each $H_\ga\subset P$ is a filter generic over the ground model containing the condition $p$;
\item $\tau_1/H_\ga$ is the same for all $\ga\in\kappa$;
\item for disjoint finite sets $a, b\subset\kappa$, $V[\tau_0/H_\ga\colon \ga\in a]\cap V[\tau_0/H_\ga\colon \ga\in b]=V$.
\end{enumerate}
\end{definition}

\begin{proposition}
\label{duplicationproposition}
Suppose that $P$ is a poset, $\tau_0, \tau_1$ are $P$-names for subsets of the ground model, and $\tau_0$ is duplicable over $\tau_1$. Then $P$ forces $V[\tau_0]$ to be Noetherian to $V[\tau_1]$ over $V$.
\end{proposition}

\begin{proof}
Suppose towards a contradiction that the conclusion fails. Then there must be a condition $p\in P$, Polish spaces $X, Y$, a Noetherian subbasis $C\subset Y\times X$ coded in $V$, and a name $\gs$ for a subset of $X$ in the model $V[\tau_1]$ such that $p$ forces $C(V[\tau_0], \gs)$ to be strictly smaller than $C(V, \gs)$. Basic forcing theory shows that $V[\tau_1]$ is a generic extension of the ground model, and that we may assume that there is a poset $Q_1$ such that $p\Vdash\tau_1\subset \check Q_1$ is a filter generic over the ground model. We also may assume that $\gs$ is in fact a $Q_1$-name.

Move to a generic extension $V[K]$ in which a sequence $\langle H_\ga\colon\ga\in\kappa\rangle$ satisfies the items of Definition~\ref{duplicationdefinition}. Write $A\subset X$ for the common value of $\gs/(\tau_1/H_\ga)$ for all ordinals $\ga\in\kappa$ and write $C_\ga=C(V[H_\ga], A)$.  For each ordinal $\gb\in\gw_1$ let $D_\gb=\bigcap_{\ga\in\gw_1\setminus\gb}C_\gb$. The sequence $\langle D_\gb\colon \gb\in\gw_1\rangle$ is an inclusion-increasing uncountable sequence of closed subsets of $X$, and as such it has to stabilize at some ordinal $\gb_0$. Write $D\subset X$ for the stable value. Use the Noetherian property of the subbasis $C$ to find  a finite set $b_0\subset\gw_1\setminus\gb_0$ such that $D=\bigcap_{\ga\in b_0}C_\ga$. Let $\gb_1\in\gw_1$ be an ordinal larger than $\max(b_0)$ and find a finite set $b_1\subset\gw_1\setminus\gb_1$ such that $D=\bigcap_{\ga\in b_1}C_\ga$. By the intersection assumption on the sequence of the generic extensions, $V[G_\ga\colon\ga\in b_0]\cap V[G_\ga\colon\ga\in\gb_1]=V$ must hold. Thus, the set $D$ is in fact coded in $V$. Now, for every ordinal $\ga\in\gw_1\setminus\gb_0$ it is the case that $A\subseteq D\subseteq C_\ga$. The definition of the set $C_\ga$ then shows that $D=C_\ga$ and the proposition follows.
\end{proof}

\section{Examples I}
\label{example1section}

This section contains several example of mutually duplicable names in forcing, which by Proposition~\ref{duplicationproposition} always lead to mutually Noetherian pairs of extensions. First, a warm-up example.

\begin{example}
If $Q_0, Q_1$ are posets with the respective names $\tau_0, \tau_1$ for their generic filters and $P=Q_0\times Q_1$, then $\tau_0, \tau_1$ are mutually duplicable in $P$. To see that $\tau_0$ is duplicable over $\tau_1$, let $\kappa$ be the successor of the maximum of $|Q_0|$ and $|Q_1|$. For every condition $p=\langle q_0, q_1\rangle\in P$, consider the product forcing of $\kappa$-many copies of $Q_0\restriction q_0$ and a single copy of $Q_1\restriction q_1$, yielding filters $G_0^\ga\subset Q_0$ and $G_1\subset Q_1$. Let $H_\ga=G_0^\ga\times H_1$ and observe that for finite disjoint sets $a, b\subset\kappa$, the intersection $V[G_0^\ga\colon \ga\in a]\cap V[G_0^\ga\colon \ga\in b]$ is equal to $V$ by the product forcing theorem.
\end{example}

\noindent The following two classes of examples deal with Cohen forcing. Recall that if $X$ is a Polish space then $P_X$ is the poset of nonempty open subsets of $X$ ordered by inclusion; it adds a single \emph{Cohen generic} point $\dotxgen$ which belongs to all sets in the generic filter. If $f\colon X\to Y$ is a continuous open map from $X$ to $Y$ then $P_X$ forces $\dot f(\dotxgen)$ to be $P_Y$ generic over the ground model by \cite[Proposition 3.1.1]{z:geometric}.

The first Cohen forcing example deals with product graphs. If $X_n$ for $n\in\gw$ are countable sets and $\Gamma_n$ are graphs on each, then the product $\prod_n\Gamma_n$ is the graph on $X=\prod_nX_n$ connecting sequences $x_0, x_1$ if for every $n\in\gw$, $x_0(n)\mathrel\Gamma_n x_1(n)$ holds. 

\begin{example}
Let $\Gamma_n$ be graphs on nonempty countable sets $X_n$, each of which is connected and not bipartite. Let $X=\prod_nX_n$ and $\Gamma=\prod_nX_n$. The poset $P_\Gamma$ adds a generic pair $\langle\tau_0, \tau_1\rangle\in\Gamma$. The names $\tau_0, \tau_1$ are mutually duplicable.
\end{example}

\noindent Note that the assumptions on the graphs $\Gamma_n$ are best possible. If they are not connected, then the sequence of the $\Gamma_n$-components of $\tau_0(n)$ belongs to both $V[\tau_0]$ and $V[\tau_1]$ and not to $V$. If the graphs $\Gamma_n$ are bipartite as witnessed by partitions $X_n=A_{n0}\cup A_{n1}$ for each $n\in\gw$, then the binary sequence $y\in\cantor$ defined by $y(n)=i$ if $\tau_0(n)\in A_{ni}$ belongs to $V[\tau_0]\cap V[\tau_1]$ (the same definition applied to $\tau_1$ yields the sequence $1-y$) but not to $V$. In both cases, the conclusion shows that the models $V[\tau_0]$ and $V[\tau_1]$ are forced not to be mutually Noetherian by Proposition~\ref{ahproposition}(1).

\begin{proof}
Let $P$ be the combinatorial presentation of the poset $P_X$: its conditions are finite sequences $p$ such that $\forall n\in\dom(p)$, $p(n)\in X_n$, ordered by reverse inclusion. The following technical claim uses the full strength of initial assumptions on the graphs:

\begin{claim}
\label{technical}
Let $p\in P$ and let $\gs$ be a $P$-name for a set of ordinals not in the ground model. Then there is an ordinal $\ga$ and conditions $p_0, p_1\leq p$ such that

\begin{enumerate}
\item $\dom(p_0)=\dom(p_1)$, there is a unique $n\in\dom(p_0)$ such that $p_0(n)\neq p_1(n)$, and there is $x\in X_n$ $\Gamma_n$ related to both $p_0(n)$ and $p_1(n)$;
\item $p_0\Vdash\check\ga\in\gs$ and $p_1\Vdash\check\ga\notin\gs$.
\end{enumerate}
\end{claim}

\begin{proof}
First, since $\gs$ is forced not to be in the ground model, there must be an ordinal $\ga$ and conditions $p_0^0, p_0^1\leq p$ with the same domain $m\in\gw$ such that $p_0^0\Vdash\check\ga\in\gs$ and $p_0^1\Vdash\check\ga\notin\gs$. The next move is to find such two conditions with the same domain and so that they differ in exactly one entry.

For this, for each condition $q\leq p_0^0$ and every number $k\leq m$, write $q_n$ for the condition obtained from $q$ by rewriting the entry $q(k)$ by $p_0^1(k)$ for every $k\in n$. Strengthening repeatedly the condition $p_0^0$, find $q\leq p_0^0$ such that for every number $n\leq m$ the condition $q_n$ decides the statement $\check\ga\in\gs$. Since $q_0\Vdash\check\ga\in\gs$ and $q_m\Vdash\check\ga\notin\gs$, there must be a number $n\in m$ such that $q_n\Vdash\check \ga\in\gs$ and $q_{n+1}\Vdash\check\ga\notin\gs$. The conditions $q_n$ and $q_{n+1}$ are as required. 

Thus, there must be an ordinal $\ga$, a number $n\in\gw$, and conditions $p_0^0, p_0^1\leq p$ with the same domain $m\in\gw$ such that $p_0^0\Vdash\check\ga\in\gs$ and $p_0^1\Vdash\check\ga\notin\gs$ and $p_0^0, p_0^1$ differ only at entry $n$. Now, let $\{x_i\colon i\in\gw\}$ be an enumeration of $X_n$ with possible repetitions. By recursion on $i\in\gw$ build conditions $q_i\in P$ such that

\begin{itemize}
\item $p_0^0=q_0\geq q_1\geq\dots$;
\item writing $q'_i$ for the condition obtained from $q_i$ by replacing its $n$-th entry with $x_i$, $q'_i$ decides the membership of $\ga$ in $\gs$.
\end{itemize}

\noindent Let $a=\{x_i\colon q'_i\Vdash\check\ga\in\gs\}$ and $b=\{x_i\colon q'_i\Vdash\check\ga\notin\gs\}$. This is a partition of $X_n$ into two nonempty sets by the choice of the condition $p_0^0$. Since the graph $\Gamma_n$ is not bipartite, one of them (say $a$) must contain an edge. Let $c\subseteq a$ be a maximal $\Gamma_n$-connected component of $a$ containing more than one element. Since the graph $\Gamma_n$ is connected, there must be a point $x_j$ which a $\Gamma_n$-neighbor of $c$, not in $c$ itself. Then $x_j\in b$ and there must be elements $x_k, x_l\in c$ such that $x_l\mathrel\Gamma_n x_k\mathrel\Gamma_n x_j$.

Let $i=\max\{j, l\}$ and consider the conditions $p_0, p_1$ obtained from $q_i$ by replacing its $n$-th entry with $x_l$ and $x_j$ respectively. The conditions $p_0, p_1$ and the ordinal $\ga$ are as required.
\end{proof}

Now, let $a$ be a finite set, and consider the posets $Q_a$ and $R_a$. $Q_a$ is a set of tuples $q$ such that for some $m=m_q$, $q\colon a\to P$ is a function and for every $i\in a$, $\dom(q(i))=m$, and for every $n\in m$ there is a point $x\in X_n$ which is $\Gamma_n$ related to all points $q(i)(n)$ for all $i\in a$. The ordering is coordinatewise inclusion. The poset $Q_a$ adds a generic point in $X^a$ which is the coordinatewise union of the conditions in the generic filter. $R_a$ is the collection of tuples $r=\langle p_r, q_r\rangle$ where $p_r\in P$ is a tuple with domain $m_r$, $q_r\in Q_a$ is a condition with $m_{q_r}=m_r$, and for each $n\in m_r$ and each $i\in a$, $p_r(n)\mathrel\Gamma_n q_r(i)(n)$ holds. The ordering in $r$ is coordinatewise again. The poset $R_a$ adds points $\tau\in X$ and $\gs_a\in X^a$ which are the union of first and coordinatewise union of second conditions in the generic filters respectively.

\begin{claim}
\label{utclaim}
$R_a$ forces the following:

\begin{enumerate}
\item $\tau$ is $P_X$-generic over the ground model and $\gs_a$ is $Q_a$-generic over the ground model;
\item if $b\subset a$ then $\langle\tau, \gs_a\restriction b\rangle$ is $R_b$-generic over the ground model;
\item $V[\tau]\cap V[\gs_a]=V$.
\end{enumerate}
\end{claim}

\begin{proof}
The first two items are elementary density arguments. For the last item, suppose that $r\in R$ is a condition, $\eta$ is a $Q_a$-name for a new set of ordinals and $\theta$ is a $P_X$-name for a new set of ordinals. I must find a condition $s\leq r$ and an ordinal $\ga$ such that $s\Vdash\check\ga\in (\eta/\gs_a)\Delta (\theta/\tau)$.

To this end, first find an ordinal $\ga$ and conditions $p_0, p_1\leq p_r$ in the poset $P$ as in Claim~\ref{technical}. Let $n\in\gw$ be the unique number such that $p_0(n)\neq p_1(n)$, and let $x\in X_n$ be a point which is $\Gamma_n$-related to both $p_0(n)$ and $p_1(n)$.
Let $q'\leq q_r$ be a condition in $Q_a$ such that for each $i\in a$ the domain of $q'(i)$ is equal to that of $p_0$, for each $m\in\dom(q'(i))$ the point $q'_i(m)$ is $\Gamma_m$-related to $p_0(m)$, and $q'_i(n)=x$. Let $q''\leq q'$ be a condition in $Q_a$ deciding the membership of $\ga$ in $\eta$; for definiteness, assume that $q''\Vdash\check\ga\in\eta$. It is not difficult to see that there is a condition $s\leq r$ such that $p_s\leq p_1$ and $q_s=q''$. Such a condition forces $\check\ga\in\eta/\gs_a$ and $\check\ga\notin\theta/\tau$ as desired.
\end{proof}

\noindent Finally, the ground is ready for a quick proof of the example. Let $S$ be the combinatorial presentation of the poset $P_\Gamma$: it consists of pairs $s=\langle u_s, t_s\rangle$ where $u_s, t_s\in P$ have the same domain and for every $n\in\dom(u_s)$, $t_s(n)\mathrel\Gamma_n u_s(n)$. The ordering is that of reverse coordinatewise inclusion. Now, suppose that $s\in S$ is a condition. Let $x_0\in X$ be $P$-generic over the ground model, extending $u_q$. In the model $V[x_0]$, consider the poset $T$ of all conditions $t\in P$ such that $t_q\subset t$ and for all $n\in\dom(t)$, $t(n)\mathrel\Gamma_n x_0(n)$. The ordering is that of inclusion. Force with the finite support product of $\gw_1$ many copies of $T$, resulting in points $x_{1\ga}\in X$ for $\ga\in\gw_1$.

A brief density argument shows that each pair $\langle x_0, x_{1\ga}\rangle$ is $S$-generic over the ground model, and for every finite set $a\subset\gw_1$, the pair $\langle x_0, \langle x_{1\ga}\colon\ga\in a\rangle\rangle$ is $R_a$-generic over the ground model. Let $a, b\subset\gw_1$ be disjoint finite sets. By a mutual genericity argument, $V[x_0][x_{1\ga}\colon \ga\in a]\cap V[x_0][x_{1\ga}\colon\ga\in b]\subseteq V[x_0]$ must hold. Thus, $V[x_{1\ga}\colon \ga\in a]\cap V[x_{1\ga}\colon\ga\in b]\subseteq V[x_0]$. By Claim~\ref{utclaim}(3), $V[x_{1\ga}\colon \ga\in a]\cap V[x_{1\ga}\colon\ga\in b]=V$ as desired.
\end{proof}

\begin{corollary}
\label{graphcorollary}
Let $\Gamma_n$ be graphs on nonempty countable sets $X_n$, each of which is connected and not bipartite. Let $X=\prod_nX_n$ and $\Gamma=\prod_nX_n$. The poset $P_\Gamma$ adds a generic pair $\langle\tau_0, \tau_1\rangle\in\Gamma$. It forces $V[\tau_0]$, $V[\tau_1]$ to be mutually Noetherian extensions.
\end{corollary}

\noindent The second Cohen forcing example deals with turbulent actions of Polish groups as outlined in \cite[Section 13.1]{kanovei:book}.

\begin{example}
\label{turbulentexample}
Let $\Gamma$ be a Polish group acting turbulently on a Polish space $X$ with all orbits dense and meager. Let  $P=P_\Gamma\times P_X$, let $\dotggen$ be the $P$-name for the $P_\Gamma$-generic point, $\tau_0$ the $P$-name for the $P_X$-generic point, and $\tau_1=\dotggen\cdot\tau_0$. Then $\tau_0, \tau_1$ are mutually duplicable names in $P$.
\end{example}

\begin{proof}
By a symmetry argument, it is enough to show that $\tau_1$ is duplicable over $\tau_1$. Suppose that $p=\langle U, O\rangle$ is a condition in the poset $P$, where $U\subset \Gamma$ and $O\subset X$ are nonempty open sets. Let $\kappa=\gw_1$, let $x\in O$ be a point $P_X$-generic over $V$. Force with a finite support product of $\kappa$-many copies of the poset $P_\Gamma\restriction U$ to obtain a sequence $\langle g_\ga\colon \ga\in\kappa\rangle$ of points in $U$ which are in finite tuples mutually $P_\Gamma$-generic over $V[x]$; write $x_\ga=g_\ga\cdot x$. Each of the filters $H_\ga\subset P$ for $\ga\in\kappa$ given by the pair $\langle g_\ga, x\rangle$ is $P$-generic over $V$; I will show that the sequence $\langle H_\ga\colon\ga\in\kappa\rangle$ witnesses the duplicability of $\tau_1$ over $\tau_0$.

\begin{claim}
Let $a\subset \kappa$ be a finite set. Then $V[x_\ga\colon \ga\in a]\cap V[x]=V$.
\end{claim}

\begin{proof}
Without loss, assume that the set $a$ is nonempty, and write $\gb=\min(a)$. The point $g_\gb^{-1}$ is $P_\Gamma$-generic over $V[x_\gb]$, and the tuple $t=\langle g_\ga\cdot g_{\gb}^{-1}\colon \gb\in a\setminus \{\gb\}\rangle$ is generic over $V[x_\gb][g_\gb^{-1}]$ for a product of the posets $P_\Gamma$. Use the product forcing theorem to conclude that $V[x_\gb][g_\gb^{-1}]\cap V[x_\gb][s]=V[x_\gb]$ and therefore $V[x]\cap V[x_\ga\colon\ga\in a]\subseteq V[x_\gb]$. However, $V[x]\cap V[x_\gb]=V$ holds by the turbulence assumption and \cite[Theorem 3.2.]{z:geometric}. It follows that $V[x]\cap V[x_\ga\colon \ga\in a]=V$ as desired.
\end{proof}

\begin{claim}
If $a, b\subset\kappa$ are disjoint finite sets then $V[x_\ga\colon\ga\in a]\cap V[x_\ga\colon \ga\in b]=V$.
\end{claim}

\begin{proof}
By the product forcing theorem, $V[x][g_\ga\colon\ga\in a]\cap V[x][g_\ga\colon \ga\in b]=V[x]$ holds, and therefore $V[x_\ga\colon\ga\in a]\cap V[x_\ga\colon\ga\in b]\subseteq V[x]$ holds. By the previous claim, $V[x_\ga\colon \ga\in a]\cap V[x]=V$ holds, and in consequence $V[x_\ga\colon\ga\in a]\cap V[x_\ga\colon\ga\in b]=V$ as desired.
\end{proof}

\noindent The duplicability follows.
\end{proof}

\begin{corollary}
The poset $P$ of Example~\ref{turbulentexample} forces $V[\tau_0]$, $V[\tau_1]$ to be mutually Noetherian extensions.
\end{corollary}

\noindent The following example uses the notion of Suslin forcing which occurs several times in this paper.

\begin{definition}
\cite[Definition 3.6.1]{bartoszynski:set}
A forcing $Q$ is \emph{Suslin} if there s an \emph{ambient} Polish space $X$ in which the conditions of $Q$ form an analytic set, and the ordering and incompatibility relations on $Q$ are analytic relations on $X$.
\end{definition}

\noindent The following fact regarding Suslin forcing will be used repeatedly throughout the paper.

\begin{fact}
\label{suslinfact}
Let $Q$ be a c.c.c.\ Suslin forcing and let $V[G]$ be a forcing extension. Then

\begin{enumerate}
\item \cite[Theorem 3.6.6]{bartoszynski:set} the reinterpretation $Q^{V[G]}$ is a c.c.c.\ Suslin forcing in $V[G]$;
\item \cite[Corollary 3.6.5]{bartoszynski:set} if $H\subset Q^{V[G]}$ is a filter generic over $V[G]$, then $H\cap V\subset Q^V$ is a filter generic over $V$.
\end{enumerate}
\end{fact}

\begin{example}
\label{suslinexample}
Let $Q_0$ be an arbitrary forcing and $Q_1$ be a Suslin c.c.c.\ forcing. Let $P$ be the iteration $Q_0*\dot Q_1$ where the definition of $Q_1$ is reinterpreted in the $Q_0$-forcing extension. Let $\tau_0$ be the $P$-name for the filter on the first iterand and $\tau_1$ be the $P$-name for the intersection of the filter on the second iterand with the ground model. Then $\tau_0, \tau_1$ are mutually duplicable names in $P$.
\end{example}

\noindent Note that by Fact~\ref{suslinfact}(2) $\tau_1$ is forced to be a filter on $Q_1$ generic over $V$.

\begin{proof}
To show that $\tau_1$ is duplicable over $\tau_0$, let $p=\langle q_0, \dot q_1\rangle$ be a condition in the poset $P$. Let $\kappa$ be a regular cardinal larger than $|Q_0|$. Let $G_0\subset Q_0$ be a filter generic over $V$ containing the condition $q_0$, and force with the finite support product of $\kappa$-many copies of the forcing $Q_1\restriction \dot q_1/G_0$ to obtain filters $G_{1\ga}\subset Q_1$ for $\ga\in\kappa$. I claim that the filters $H_\ga= G_0*G_{1\ga}\subset P$ for $\ga\in\kappa$ witness the duplicability of $\tau_1$ over $\tau_0$.
To see this, for each $\ga\in\kappa$ write $K_\ga=G_{1\ga}\cap V=\tau_1/H_\ga$. The following claim completes the proof by the product forcing theorem.

\begin{claim}
If $a\subset\kappa$ is a finite set, then $\langle K_\ga\colon \ga\in a\rangle$ are filters on $Q_1$ which are mutually generic over $V$.
\end{claim}

\begin{proof}
Consider the poset $R$ which is the product of $a$-many copies of $Q_1$. It is easy to check that $R$ is Suslin, and by Fact~\ref{suslinfact}(1), $R$ is c.c.c. The filters $\langle G_{1\ga}\colon \ga\in a\rangle$ form an $R$-generic sequence over $V[G_0]$. By Fact~\ref{suslinfact}(2), the sequence $\langle K_\ga\colon\ga\in a\rangle$ is an $R$-generic sequence over $V$ as desired.
\end{proof}

To show that $\tau_0$ is duplicable over $\tau_1$,  let $p=\langle q_0, \dot q_1\rangle$ be a condition in the poset $P$. Let $\kappa$ be a regular cardinal larger than $|Q_0|$. Let $s=\langle G_{0\ga}\colon \ga\in\kappa\rangle$ be a mutually generic sequence of filters on $Q_0\restriction q_0$. In the model $V[s]$, consider the (reinterpretation of the) poset $Q_1$ and the conditions $r_\ga=\dot q_1/G_{0\ga}$ in it for $\ga\in\kappa$. Since $Q_1$ is c.c.c.\ in $V[s]$ by Fact~\ref{suslinfact}(1) there must be a condition $r\in Q_1$ which forces that the set $\{\ga\in\kappa\colon r_\ga$ belongs to the generic filter$\}$ is cofinal in $\kappa$. Let $G_1\subset Q_1$ be a filter generic over $V[s]$, containing the condition $r$. By Fact~\ref{suslinfact}(2), for each ordinal $\ga\in\kappa$, $G_{1\ga}=G_1\cap V[G_{0\ga}]$ is a filter on $Q_1$ generic over $V[G_{0\ga}]$. I claim that the filters $H_\ga=G_{0\ga}*G_{1\ga}$ for $\ga\in\kappa$ such that $r_\ga\in G_1$ witness the duplicability of $\tau_0$ over $\tau_1$. This is an immediate corollary of the product forcing theorem applied to the models $V[G_{0\ga}]$ for $\ga\in\kappa$.
\end{proof}

\begin{corollary}
Let $\mu$ be a Borel probability measure on a Polish space $X$, and let $x_0, x_1\in X$ be mutually $\mu$-random elements of $X$. Then the models $V[x_0]$, $V[x_1]$ are mutually Noetherian.
\end{corollary}

\section{Preservation theorems}
\label{preservationsection}

Any notion of perpendicularity similar to Definition~\ref{extdefinition} comes with a natural notion of balance for analytic forcings.

\begin{definition}
Let $P$ be an analytic forcing.

\begin{enumerate}
\item A pair $\langle Q, \gs\rangle$ is \emph{Noetherian balanced} if $Q\Vdash\gs\in P$ and for any pair $V[G_0]$, $V[G_1]$ of mutually Noetherian extensions of the ground model, every pair $H_0, H_1\subset Q$ of filters generic over $V$ and for every pair $p_0\in V[G_0]$, $p_1\in V[G_1]$ of conditions stronger than $\gs/H_0, \gs/H_1$ respectively and belonging to the respective models $V[G_0], V[G_1]$, the conditions $p_0, p_1\in P$ have a common lower bound.
\item $P$ is \emph{Noetherian balanced} if for every condition $p\in P$ there is a Noetherian balanced pair $\langle Q, \gs\rangle$ such that $Q\Vdash\gs\leq\check p$.
\end{enumerate}

\noindent Similar definitions apply for $K_\gs$-Noetherian.
\end{definition}

\noindent Note that if a poset is $K_\gs$-Noetherian balanced, then it is Noetherian balanced, since the key quantification over pairs of generic extensions includes more of them in case of mutually $K_\gs$-Noetherian pairs.

The supply of mutually Noetherian pairs of extensions provided in the previous section now makes it possible to prove several preservation theorems. They are stated using the parlance of \cite[Convention 1.7.18]{z:geometric}. Thus, given an inaccessible cardinal $\kappa$, a Suslin poset $P$ is Noetherian balanced cofinally below $\kappa$ if for every generic extension $V[K_0]$ generated by poset of cardinality smaller than $\kappa$ there is a larger generic extension $V[K_1]$ generated by a poset of cardinality smaller than $\kappa$ such that $V_\kappa[K_1]\models P$ is Noetherian balanced. 

\begin{theorem}
\label{xitheorem}
In cofinally Noetherian balanced extensions of the choiceless Solovay model, if $A_0, A_1\in 3^\gw$ are non-meager sets then there are points $y_0\in A_0$ and $y_1\in A_1$ such that the set $\{i\in\gw\colon y_0(i)=y_1(i)\}$ is finite.
\end{theorem}

\noindent This proves Theorem~\ref{firsttheorem}(1). I do not know if the conclusion holds also for non-null sets for the usual Borel probability measure on $3^\gw$.

\begin{proof}
Let $\kappa$ be an inaccessible cardinal. Let $P$ be a Suslin forcing which is Noetherian balanced cofinally below $\kappa$. Let $W$ be the choiceless Solovay model derived from $\kappa$. Work in $W$. Let $p\in P$ be a condition and let $\tau_0, \tau_1$ be $P$-names such that $p\Vdash\tau_0, \tau_1\subset 3^\gw$ are non-meager sets. I have to find two points $y_0, y_1$ such that the set $\{i\in\gw\colon y_0(i)=y_1(i)\}$ is finite and a condition stronger than $p$ which forces them to $\tau_0, \tau_1$ respectively.

Both $p, \tau_0, \tau_1$ are definable from some elements of the ground model and an additional parameter $z\in\cantor$. Let $V[K]$ be an intermediate extension obtained by a partial order of cardinality less than $\kappa$ such that $z\in V[K]$, and such that $V[K]\models P$ is Noetherian balanced. Work in $V[K]$. Let $\langle Q, \gs\rangle$ be a Noetherian balanced pair such that $Q\Vdash\gs\leq p$. Let $R$ be the Cohen poset of nonempty open subsets of $3^\gw$ ordered by inclusion, adding a Cohen generic point $\dot y$. 

For each $i\in 2$, there has to be a poset $S_i$ of cardinality smaller than $\kappa$, a $Q\times R\times S_i$-name $\eta_i$ for a condition in $P$ stronger than $\gs$, and conditions $q_i\in Q$, $r_i\in R$ and $s_i\in S_i$ which force in the product $Q\times R\times S_i$ the following.  In the $\coll(\gw, <\kappa)$-extension, in the poset $P$, $\eta_i\Vdash\dot y\in\tau_i$ holds.
Otherwise, in the model $W$, if $H\subset Q$ is a filter generic over $V[K]$ then the condition $\gs/H$ would force in $P$ that the comeager set of points $R$-generic over $V[K][H]$ to be disjoint from $\tau_i$, contradicting the initial assumptions on $p$ and $\tau_i$.

In the model $W$, use Corollary~\ref{graphcorollary} to produce points $x_0, x_1\in 3^\gw$ which are are separately $R$-generic over $V[K]$, such that for all $i\in\gw$ $x_0(i)\neq x_1(i)$, and such that the models $V[K][x(0)]$, $V[K][x_1]$ are mutually Noetherian. Let $y_0\in 3^\gw$ be a finite modification of $x_0$ which belongs to $r_0$ and let $y_1\in 3^\gw$ be a finite modification of $x_1$ which belongs to $r_1$. Let $H_0\subset Q\times S_0$, $H_1\subset Q\times S_1$ be filters mutually generic over $V[K][x_0][x_1]$ meeting the respective conditions $q_0, q_1\in Q$ and $s_0\in S_0$ and $s_1\in S_1$. By Proposition~\ref{productproposition}, the models $V[K][y_0][H_0]$ and $V[K][y_1][H_1]$ are mutually Noetherian extensions of $V[K]$. Let $p_0=\gs/y_0, H_0$ and $p_1=\gs/y_1, H_1$. These are conditions in $P$ in the respective models stronger than $\gs/H_0$ and $\gs/H_1$ respectively. By the balance assumption on the pair $\langle Q, \gs\rangle$, the conditions $p_0, p_1$ are compatible. By the forcing theorem applied in the respective models $V[K][y_0][H_0]$ and $V[K][y_1][H_1]$, the common lower bound of these two conditions forces $\check y_0\in\tau_0$ and $\check y_1\in\tau_1$ as required.
\end{proof}

\begin{corollary}
\label{ultrafiltercorollary}
In cofinally Noetherian balanced extensions of the choiceless Solovay model, there are no non-principal ultrafilters on $\gw$.
\end{corollary}

\begin{proof}
If $U$ is a non-principal ultrafilter on $\gw$, then the map $c\colon 3^\gw\to 3$ defined by $c(x)=i$ if $\{n\in\gw\colon c(n)=i\}\in U$ partitions of $3^\gw$ into three pieces neither of which contains points $y_0, y_1$ such that the set $\{i\in\gw\colon y_0(i)=y_1(i)\}$ is finite. One of these pieces must be non-meager. Theorem~\ref{xitheorem} concludes the proof.
\end{proof}

\begin{theorem}
\label{ttheorem}
Let $\Gamma\curvearrowright X$ be a turbulent action of a Polish group with all orbits dense and meager.  In cofinally Noetherian balanced extensions of the choiceless Solovay model, if $A_0, A_1\subset X$ are non-meager sets, then there are points $x_0\in A_0$, $x_1\in A_1$, and $\Gamma\in\Gamma$ such that $\gamma\cdot x_0=x_1$.
\end{theorem}

\noindent This proves Theorem~\ref{firsttheorem}(2).

\begin{proof}
Let $\kappa$ be an inaccessible cardinal. Let $P$ be a Suslin forcing which is Noetherian balanced cofinally below $\kappa$. Let $W$ be the symmetric Solovay model derived from $\kappa$. Work in $W$. Let $p\in P$ be a condition and let $\tau_0, \tau_1$ be $P$-names such that $p\Vdash\tau_0, \tau_1\subset X$ are both non-meager sets. I have to find an element $\gamma\in\Gamma$ and points $x_0, x_1\in X$ such that $\gamma\cdot x_0=x_1$, and a condition stronger than $p$ which forces $\check x_0\in\tau_0$ and $\check x_1\in\tau_1$.

Both $p, \tau$ are definable from some elements of the ground model and an additional parameter $z\in\cantor$. Let $V[K]$ be an intermediate extension obtained by a partial order of cardinality less than $\kappa$ such that $z\in V[K]$, and such that $V[K]\models P$ is Noetherian balanced. Work in $V[K]$. Let $\langle Q, \gs\rangle$ be a Noetherian balanced pair such that $Q\Vdash\gs\leq\check p$. Let $R$ be the Cohen poset of all nonempty open subsets of $X$, adding a point $\dotxgen\in X$. For each $i\in 2$ there must be a poset $S_i$ of cardinality smaller than $\kappa$, a $Q\times R\times S_i$-name $\eta_i$ for a condition in $P$ stronger than $\gs$, and conditions $q_i\in Q$, $r_i\in R$ and $s_i\in S_i$ forcing in the product $Q\times R\times S_i$ the following. In the $\coll(\gw, <\kappa)$-extension, in the poset $P$, $\eta_i\Vdash \dotxgen\in \tau_i$. Otherwise, in the model $W$, if $H\subset Q$ is a filter generic over the model $V[K]$, the condition $\gs/H\in P$ would force $\tau_i$ to be disjoint from the co-meager set of elements of $X$ Cohen-generic over the model $V[K][H]$. This would contradict the initial assumptions on the name $\tau$. 

Now, since the group $\Gamma$ acts on $X$ continuously and with dense orbits, thinning out the open set $r_0\subset X$ if necessary, one can find an open set $U\subset\Gamma$ such that $U\cdot r_0\subset r_1$. In the model $W$, find points $\gamma\in U$ and $x_0\in r_0$ which are $P_\Gamma\times R$-generic over $V[K]$, and let $x_1=\gamma\cdot x_0$. By the turbulence assumption and Example~\ref{turbulentexample}, the models $V[K][x_0]$ and $V[K][x_1]$ are mutually Noetherian extensions of $V[K]$. Let $H_0\subset Q\times S_0$ and $H_1\subset Q\times S_1$ be filters mutually generic over $V[K][x_0][x_1]$ and containing the respective conditions $q_0, s_0, q_1, s_1$. By Proposition~\ref{productproposition}, the models $V[K][x_0][H_0]$ and $V[K][x_1][H_1]$ are mutually Noetherian as well. 

Let $p_0=\eta_0/x_0, H_0$ and $p_1=\eta_1/x_1, H_1$. These are conditions in $P$ in mutually Noetherian extensions extending the conditions $\gs/H_0$ and $\gs/H_1$. By the balance assumption on the pair $\langle Q, \gs\rangle$, the conditions $p_0, p_1$ have a common lower bound in $P$. In the model $W$, the common lower bound forces both $\check x_0\in\tau$ and $\check x_1\notin\tau$, while it is also the case that $\gamma\cdot x_0=x_1$. This completes the proof of the theorem.
\end{proof}

\noindent For the following corollary, recall that a tournament on a set $X$ is a an orianted graph on $X$ which for any two elements of $X$ contains exactly one ordered pair consisting of the two.

\begin{corollary}
\label{tcorollary}
Let $E$ be an orbit equivalence relation of a turbulent action of a Polish group on a Polish space $Y$. In cofinally Noetherian balanced extensions of the choiceless Solovay model, there is no tournament on the quotient space $Y/E$.
\end{corollary}

\begin{proof}
Suppose towards a contradiction that $T$ is a tournament on the $Y/E$-space. Let $\Gamma$ be a Polish group turbulently acting on $Y$, inducing the orbit equivalence relation $E$. The product action of $\Gamma\times\Gamma$ on $Y\times Y$ is easily checked to be turbulent as well, inducing the equivalence relation $E\times E$. The set $B=\{\langle y_0, y_1\rangle\in Y\times Y\colon \langle [y_0]_E, [y_1]_E\rangle\in T\}$ is $E\times E$-invariant, and therefore meager or co-meager by Theorem~\ref{ttheorem}. For definiteness, assume that $B$ is co-meager.
Then, the set $B^\perp=\{\langle y_0, y_1\rangle\in Y\times Y\colon \langle y_1, y_0\rangle\in B\}$ is co-meager as well. Let $\langle y_0, y_1\rangle$ be a pair in the intersection of the two co-meager sets $B$ and $B^\perp$. It follows that both $\langle [y_0]_E, [y_1]_E\rangle$ and $\langle [y_1]_E, y_0]_E\rangle$ should be in the tournament $T$, an impossibility.
\end{proof}

\begin{theorem}
In cofinally Noetherian balanced $\gs$-closed extensions of the choiceless Solovay model, outer Lebesgue measure on $[0, 1]$ is continuous in arbitrary increasing unions.
\end{theorem}

\noindent That is to say, writing $X=[0,1]$ and $\lambda$ for the Lebesgue measure on $X$ and $\lambda^*$ for the outer Lebesgue measure, if $A\subset\power(X)$ is a collection of sets linearly ordered by inclusion, then $\lambda^*(\bigcup A)=\sup_{B\in A}\lambda^*(B)$. This proves Theorem~\ref{firsttheorem}(3). One should note that in most Noetherian balanced extensions, there are non-measurable sets; the conclusion of the theorem does not rule these out.

\begin{proof}
 Let $\kappa$ be an inaccessible cardinal. Let $P$ be a $\gs$-closed Suslin forcing which is Noetherian balanced cofinally below $\kappa$. Let $W$ be the choiceless Solovay model derived from $\kappa$. Work in $W$.

Suppose towards a contradiction that $p\in P$ is a condition, $\eps$ is a real number, $\tau$ is a $P$-name, and the condition $p$ forces $\tau$ to be a collection of subsets of $X$ which is linearly ordered by inclusion, each element of $\tau$ has outer Lebesgue measure smaller than $\eps$, and $\lambda^*(\bigcup\tau)>\eps$.

\begin{claim}
There is an open set $O\subset X$ such that $\lambda(O)<\eps$ and a condition stronger than $p$ forcing $\forall a\in\tau\ \forall U (\lambda(O\cap U)<\lambda(U)/4\to\lambda^*(a\cap U)<\lambda(U)/4$.
\end{claim}

\begin{proof}
Using DC in $W$, find a descending sequence $\langle p_n\colon n\in\gw\rangle$ starting with $p=p_0$ and a sequence $\langle \dot a_n\colon n\in\gw\rangle$ such that for every $n\in\gw$, $p_{n+1}$ forces that $\dot a_n\in\tau$ and if there is an $a\in\tau$ such that $\lambda^*(a\cap U_n)\geq 1/4$ then $\dot a_n$ is such. Use the $\gs$-closure assumption on the poset $P$ to find a lower bound $p_\gw$ of the sequence. 

Now, $p_\gw\Vdash\{\dot a_n\colon n\in\gw\}$ is not a cofinal subset of $\tau$. The reason is that $\lambda^*$ is continuous in increasing unions of countable length. Therefore, strengthening $p_\gw$ if necessary, it is possible to find a name $\dot a_\gw$ and an open set $O\subset X$ such that $\lambda(O)<\eps$ and $p_{\gw}$ forces $\dot a_\gw\in\tau$, $\bigcup_n\dot a_n\subset\dot a_\gw$, and $\dot a_\gw\subseteq O$. The condition $p_\gw$ and the set $O\subset X$ are as required.
\end{proof}

Let $O\subset X$ be an open set as in the claim, and strengthen the condition $p$ as to force the conclusion of the claim. Now, the name $\tau$ is definable from a parameter $z\in\cantor$ and additional parameters from the ground model. Let $V[K]$ be an intermediate extension containing $z$, $p$, and $O$ and such that $P$ is balanced in $V[K]$. In the model $V[K]$, let $\langle Q, \gs\rangle$ be a Noetherian balanced pair in the poset $P$ such that $Q\Vdash\gs\leq p$. 

The following considerations take place in the model $W$, using the fact that every set is Lebesgue measurable in $W$.
Let $B=\{x\in X\colon\exists G\subset Q$ generic over $V[K]\ \exists p\leq\gs/G\ p\Vdash x\in\bigcup\tau\}$. 
Note that $\lambda(B)\geq\eps$ since whenever $G\subset Q$ is generic over $V[K]$, then the condition $\gs/G$ forces $\bigcup\tau$ to be a subset of $B$. Thus, $\lambda(B\setminus O)>0$ and by the Lebesgue density theorem, there must be a basic open set $U\subset X$ such that $\lambda(B\setminus O\cap U)>3\lambda(U)/4$. Note that $p$ forces that every element of $\tau$ has relative outer measure in $U$ smaller than $1/4$ by the choice of the open set $O$.

Let $C_0=\{\langle x_0, x_1\rangle\in U\times U\colon\exists G_0\subset Q$ generic over $V[K]\ \exists p_0\leq\gs/G_0\ \exists O_0\subset X$ of relative open measure in $U$ smaller than 1/4 such that $x_1\notin O_0$ is random over $V[x_0][G_0, p_0, O_0]$ and $p_0\Vdash x_0$ belongs to some element of $\tau$ which is a subset of $O_0$. 

\begin{claim}
The relative product measure of $C_0$ in $U\times U$ is greater than $1/2$.
\end{claim}

\begin{proof}
Use the Fubini theorem. Let $x_0\in B\cap U\setminus O$ be a point random over $V[K]$--there are relative measure $>3/4$ of such points in the set $U$. Let $G_0, p_0$ witness the membership of $x_0$ in $B$, and strengthen $p_0$ if necessary to find an open set $O_0\subset X$ of relative measure in $U$ smaller than $1/4$ such that $p_0\Vdash\check x_0$ belongs to some element of $\tau$ which is a subset of $O_0$. Now, any point $x_1\in U$ which is random over $V[K][x_0, G_0, p_0, O_0]$ and does not belong to $O_0$ (and there is more than relative measure $3/4$ of these) belongs to the vertical section of $C_{x_0}$. The relative product measure of $C$ is thus greater than $3/4\cdot 3/4=1/2$ as desired.
\end{proof}

Let $D=\{\langle x_0, x_1\rangle\in U\times U\colon \langle x_0, x_1\rangle\in C$ and $\langle x_1, x_0\rangle\in C\}$. By the claim, the set $D$ has positive product measure. Working in the model $W$, let $\langle x_0, x_1\rangle\in D$ be a pair random over $V[K]$. In $V[K][x_0, x_1]$, there must be a poset $R_0$ of cardinality smaller than $\kappa$ adding witnesses to the statement that $\langle x_0, x_1\in C_0$. There also must be a poset $R_1$ of cardinality smaller than $\kappa$ adding witnesses to the statement that $\langle x_1, x_0\rangle\in C$. Let $H_0\subset R_0$ and $H_1\subset R_1$ be filters mutually generic over $V[K][x_0, x_1]$. Let $G_0, p_0, O_0$ be the witnesses to $\langle x_0, x_1\rangle$ added by $H_0$. Let $G_1, p_1, O_1$ be the witnesses to $\langle x_1, x_0\rangle$ added by $H_1$. 

Note that $V[K][x_0, G_0, p_0, O_0]$ and $V[K][x_0, x_1]$ are mutually Noetherian over $V[K][x_0]$, and $V[K][x_0, G_0, p_0, O_0]$ and $V[K][x_1]$ are mutually Noetherian over $V[K]$ by Example~\ref{suslinexample} applied in the model $V[K]$ and $V[K][x_0]$ respectively to the Suslin c.c.c.\ poset of Borel non-null sets ordered by inclusion. The same assertions are true with $0,1$ interchanged.
Proposition~\ref{overlapproposition} shows that the models $V[K][x_0, G_0, p_0, O_0]$ and $V[K][x_1, G_1, p_1, O_1]$ are mutually Noetherian. 

The balance assumption on the pair $\langle Q, \gs\rangle$ now shows that the conditions $p_0, p_1$ are compatible. Their common lower bound forces that there are two elements of $\tau$, one of them containing $x_0$ but not $x_1$, the other containing $x_1$ but not $x_0$. In other words, $\tau$ is not linearly ordered by inclusion, a contradiction.
\end{proof}

\begin{corollary}
In cofinally Noetherian balanced extensions of the choiceless Solovay model, the Lebesgue null ideal is closed under well-ordered unions.
\end{corollary}

\begin{proof}
Let $\langle B_\ga\colon \ga\in\kappa\rangle$ be a well-ordered sequence of Lebesgue null sets. By induction on $\ga\leq\kappa$ argue that the union of the sets on the sequence up to $\ga$ is null, using the continuity of outer Lebesgue measure under increasing unions at every stage.
\end{proof}

\section{Examples II}
\label{example2section}

In this section, I produce several classes of Noetherian balanced Suslin forcings. This will verify all items of Theorem~\ref{secondtheorem}.

\begin{example}
\label{hamelexample}
Let $X$ be a Polish field and $F$ be a countable subfield. Let $P$ be the partial order of countable subsets of $X$ which are linearly over $F$. The ordering is reverse inclusion \cite[Example 6.3.6]{z:geometric}. Then $P$ is $\gs$-closed, Suslin and Noetherian balanced.
\end{example}

\begin{proof}
If $p\subset X$ is a Hamel basis of $X$ over $F$, $V[G_0]$ and $V[G_1]$ are generic extensions such that $V[G_0]\cap V[G_1]\cap\cantor=V\cap\cantor$ and $p_0\in V[G_0]$ and $p_1\in V[G_1]$ are conditions containing $p$ as a subset, then $p_0, p_1$ are compatible, as the inspection of the proof of \cite[Theorem 6.3.4]{z:geometric} reveals. At the same time, if $V[G_0]$ and $V[G_1]$ are mutually Noetherian extensions, then $V[G_0]\cap V[G_1]\cap\cantor=V\cap\cantor$ holds by Proposition~\ref{ahproposition}(1).
\end{proof}

\noindent It is clear that this argument extends to cover most posets which are placid in the sense of \cite[Definition 9.3.1]{z:geometric}, such as the posets adding a maximal acyclic subset to a given Borel graph \cite[Example 6.3.7]{z:geometric}, or posets adding a selector to a given countable Borel equivalence relation.

\begin{example}
\label{fieldtheorem}
Let $X$ be a Polish field and $F$ be a countable subfield. Let $P$ be the partial order of countable algebraically free over $F$ subsets of $X$, ordered by reverse inclusion \cite[Example 6.3.10]{z:geometric}. Then $P$ is $\gs$-closed, Suslin and Noetherian balanced.
\end{example}

\begin{proof}
Let $p$ be a transcendence basis for $X$ over $F$; it will be enough to show that $\langle \coll(\gw, X), \check p\rangle$ is a Noetherian balanced pair. To this end, suppose that $V[G_0]$, $V[G_1]$ are mutually Noetherian extensions, containing respective algebraically free sets $p_0, p_1$ containing $p$; it is necessary to show that $p_0\cup p_1$ is algebraically free. Towards a contradiction, suppose that this fails and let $\phi(\bar v_0, \bar v_1)$ be a nonzero polynomial with coefficients in $F$, and $\bar x_0\in p_0$ and $\bar x_1\in p_1\setminus p_0$ are tuples such that $\phi(\bar x_0, \bar x_1)=0$. By Proposition~\ref{ahproposition}(3), there is a tuple $\bar x_0'$ in $V$ such that $\phi(\bar x_0',\bar x_1)=0$ holds. Since $p$ is a transcendence basis for $X$ over $F$ in $V$, all elements of $\bar x'_0$ are algebraic over $p$. In conclusion, the tuple $\bar x_1$ satisfies a nontrivial algebraic identity over $p$, violating the assumption that $p_1$ is a set algebraically free over $F$.
\end{proof}

\begin{example}
Let $\Gamma$ be a closed Noetherian graph on a Polish space $X$ with no uncountable cliques. The coloring poset of \cite{z:ngraphs} is $\gs$-closed, Suslin and Noetherian balanced.
\end{example}

\begin{proof}
A closed graph $\Gamma$ on a Polish space $X$ is Noetherian if the set $C\subset X\times X$ consisting of all pairs $\langle x_0, x_1\rangle$ such that $x_0=x_1$ or $x_0\mathcal\Gamma x_1$ is a Noetherian subbasis. An inspection of the proof of balance in \cite{z:ngraphs} shows that the only feature of mutually generic extensions $V[G_0], V[G_1]$ used is that for every pair of $\Gamma$-connected points $x_0\in V[G_0]$ and $x_1\in V[G_1]$, there are points $x_1'\in V$ $\Gamma$-connected to $x_0$ and arbitrarily close to $x_1$. This holds already when the models $V[G_0], V[G_1]$ are mutually Noetherian. To see this, consider the set $C(V[G_0], x_1)=C(V, x_1)$. All of its points are $\Gamma$-connected (or equal) to $x_0$. If $O\subset X$ is any basic open neighborhood of $x_1$, a Mostowski absoluteness argument shows that there must be a point $x'_1\in C(V, x_1)\cap O$ in $V$ as there is such a point, namely $x_1$, in $V[G_1]$. The point $x'_1$ is as desired.
\end{proof}

\begin{example}
Let $\Gamma$ be a $\gs$-algebraic graph on a Euclidean space $X$ without uncountable cliques. The coloring poset $P$ of \cite{z:distance} is $\gs$-closed, Suslin and Noetherian balanced.
\end{example}

\begin{example}
Let $P$ be the coloring poset for the hypergraph of isosceles triangles in $\mathbb{R}^2$ isolated in \cite{zhou:isosceles}. The poset is $\gs$-closed and Noetherian balanced.
\end{example}

\begin{example}
Let $\Gamma$ be a $\gs$-algebraic redundant hypergraph on a Euclidean space $X$. The coloring poset $P$ for $\Gamma$ in \cite{z:redundant} is $\gs$-closed and Noetherian balanced.
\end{example}

\begin{proof}
An inspection of the proof of balance in the above three examples shows that the only property of generic extensions $V[G_0], V[G_1]$ used is the coheir property of Proposition~\ref{ahproposition}(3).
\end{proof}

\section{Examples III}
\label{example3section}

As a final note in this paper, it is interesting to see how the conclusions of Theorem~\ref{firsttheorem} can be violated by balanced posets which are not Noetherian balanced. The conclusions of the first two items of Theorem~\ref{firsttheorem} are easiest to violate using the following theorem.

\begin{theorem}
Let $\Gamma$ be a Borel meager graph on a Polish space $X$. In a balanced extension of the Solovay model, there are non-meager sets $A_0, A_1\subset X$ such that $(A_0\times A_1)\cap\Gamma=0$.
\end{theorem}

\begin{proof}
Without loss, assume that $\Gamma$ is $F_\sigma$. Removing the meager set of all points $x\in X$ such that the set $\{y\in X\colon x\mathrel\Gamma y\}$ is non-meager from the space $X$ if necessary, we may assume that for the set $\{y\in X\colon x\mathrel\Gamma y\}$ is meager for every $x\in X$.

Let $P$ be the partial order of all triples $p=\langle a_p, b_p, c_p\rangle$ such that $a_p, b_p, c_p\subset X$ are all countable sets, $a_p\cap b_p=0$, and $(a_p\times b_p)\cap\Gamma=0$. The ordering is defined by $q\leq p$ if $a_p\subseteq a_q$, $b_p\subseteq b_q$, $c_p\subseteq c_q$, and no element of $c_p$ is $\Gamma$-connected to any element of $(a_q\setminus a_p)\cup (b_q\setminus b_p)$. It is not difficult to check that $P$ is a $\gs$-closed Suslin poset.

\begin{claim}
$P$ is a balanced forcing.
\end{claim}

\begin{proof}
Let $p\in P$ be a condition. It will be enough to show that the pair $\langle \coll(\gw, X), \langle\check a_p, \check b_p, (X\cap V)\rangle\rangle$ is balanced in the poset $P$.

Indeed, let $V[G_0]$ and $V[G_1]$ be mutually generic extensions containing conditions $p_0=\langle a_0, b_0, c_0\rangle\in V[G_0]$ and $p_1=\langle a_1, b_1, c_1\rangle\in V[G_1]$ respectively, both stronger than $\langle a_p, b_p, (X\cap V)\rangle$. To show that $\langle a_0\cup a_1, b_0\cup b_1, c_0\cup c_1\rangle$ is a common lower bound of the conditions $p_0, p_1$, It is only necessary to observe that no element in $a_0\setminus a_p$ is $\Gamma$-connected with any element of $b_1\cup c_1$; the other verification items are proved in a symmetric way.

Thus, towards a contradiction assume that $x_0\in a_0\setminus a_p$ and $x_1\in b_1\cup c_1$ are $\Gamma$-connected. Find a closed set $C\subset\Gamma$ such that $\langle x_0, x_1\rangle\in C$. Let $e_0$ be the set of basic open sets $O\subset X$ such that $\{x_0\}\times O\cap\Gamma=0$; let $e_1$ be the set of basic open sets $O\subset X$ such that $x_1\in O_1$. The sets $e_0\in V[G_0]$ and $e_1\in V[G_1]$ are disjoint. By a mutual genericity argument \cite[Proposition 1.7.9]{z:geometric}, there are disjoint ground model sets $f_0\supset e_0$ and $f_1\supset e_1$. By a Mostowski absoluteness argument between $V[G_1]$ and $V$ there is a point $x'_1\in V$ such that every basic open neighborhood of it belongs to $f_1$--since the point $x_1\in V[G]$ is such. It follows that $\langle x_0, x'_1\rangle\in\Gamma$, and this contradicts the assumption that $p_0\leq \langle a_p, b_p, X\cap V\rangle$.
\end{proof}

\noindent If $G\subset P$ is a filter generic over the choiceless Solovay model (or indeed any other model), let $A=\{a_p\colon p\in G\}$ and $B=\{b_p\colon p\in G\}$. It is clear from the definition of the poset $P$ that $(A\times B)\cap\Gamma=0$. To complete the proof of the theorem, it is enough to argue that the sets $A, B\subset X$ are non-meager.

For this, return to the ground model, let $p\in P$ be a condition, and let $C_n\subset X$ be closed nowhere dense sets for $n\in\gw$. By a density argument, it is only necessary to find a condition $q\leq p$ such that $a_q$ contains a point in $X\setminus\bigcup_nC_n$. The arguument for $b_q$ is the same. To produce the condition $q$, note that the set $D=\{y\in X\colon \exists z\in b_p\cup c_p\ z\mathrel\Gamma y\}$ is meager by the initial assumptions on the graph $\Gamma$. Just pick a point $x\in X$ which does not belong to the meager set $\bigcup_nC_n\cup D\cup b_p\}$, let $q=\langle a_p\cup\{x\}, b_p, c_p\rangle$, and note that $q\leq p$ holds. The proof is complete. 
\end{proof}

\noindent  The conclusion of Theorem~\ref{firsttheorem}(3) is more difficult to violate.

\begin{theorem}
Let $X$ be a Polish space and $I$ a suitably definable ideal on $X$ such that every countable subset of $X$ is a subset of a $G_\gd$ set in $I$. In a balanced extension of the Solovay model, $X$ is the union of a collection of $I$-small sets, linearly ordered by inclusion.
\end{theorem}

\noindent In particular, it is possible to exhaust the unit interval $[0,1]$ as a union of Lebesgue null sets linearly ordered with respect to inclusion.

\begin{proof}
The partial order $P$ consists of conditions $p=\langle \preceq_p, b_p\rangle$ such that for some countable set $\supp(p)\subset X$, $\preceq_p$ is a linear preorder on $\supp(p)$ and $b_p$ is a countable set of pairs $\langle x, O\rangle$ where $x\in\supp(p)$ and $O\subset X$ is an open set. In addition, if $y\preceq_p x$ are points in $\supp(p)$ and $\langle x, O\rangle\in b_q$ then $y\in O$.
The ordering is defined by $q\leq p$ if $\supp(p)\subseteq\supp(q)$, $\preceq_p=\preceq_p\restriction\supp(p)$, $b_p\subseteq b_q$, and for every $y\in\supp(q)$, if $x$ is $\preceq_q$-smallest element of $\supp(p)$ such that $y\preceq_q x$, then $x\preceq_q y$.

To see the idea behind the definition, suppose that $G\subset P$ is a generic filter and let $\preceq=\bigcup_{p\in G}\preceq_p$. An elementary density argument shows that $\preceq$ is a linear preorder on all of $X$. The assumptions on the ideal $I$ imply that for every element $x\in X$, the initial segment of $\preceq$ determined by $x$ belongs to the ideal $I$: if $p\in P$ is a condition containing $x$ in its support, there is a $G_\gd$ set $\bigcap_nO_n\in I$ containing all elements of $\supp(p)$, and then the condition $q\leq p$ obtained from adding the pairs $\langle x, O_n\rangle$ for $n\in\gw$ into $b_q$ ensures that the initial segment determined by $x$ is a subset of $\bigcap_nO_n$.

It is not difficult to check that the poset $P$ is analytic and $\gs$-closed. The balance is the only hard part. Let $p\in P$ be a condition. Adding some ordered pairs to $b_p$ if necessary I may assume that if $x, y\in\supp(p)$ then

\begin{enumerate}
\item if $y\preceq_p x$ and $\langle x, O\rangle\in b_p$ then $\langle y, O\rangle\in b_p$;
\item if $x\not\preceq_p y$ then $\langle y, \cantor\setminus\{x\}\rangle\in b_p$.
\end{enumerate}

\noindent For each $y\in X$ let $a(y)=\{x\in \supp(p)\colon y\in O$ for every $O$ with $\langle x, O\rangle\in b_p\}$. Note that by for each $x\in\supp(p)$, $a(x)=[x, \infty)$ by (2) above. Define a relation $\preceq$ on $\cantor$ by $y_0\preceq y_1$ if $a_{y_1}\subseteq a_{y_0}$. Let also $b$ be the set of all pairs $\langle y, O\rangle$ such that for some $x\in a(y)$, $\langle x, O\rangle\in b_p$. 

\begin{claim}
$\coll(\gw, X)$ forces $\langle \preceq, b\rangle$ is a condition in $P$ stronger than $p$.
\end{claim}

\begin{proof}
I first show that $\langle \preceq, b\rangle$ is forced to be a condition in the poset $P$. First of all, $\preceq$ is clearly a preorder. It is linear: by (1) above each set $a(y)$ is a cofinal segment of $\preceq_p$, and such sets are linearly ordered by inclusion. Now, suppose that $y\preceq x$ and $\langle x, O\rangle\in b$. By the definitions, $\langle z, O\rangle\in b_p$ for some $z\in a(x)$. Since $a(x)\subseteq a(y)$, $z\in a(y)$ holds, so $y\in O$ holds as desired.

To show that $\langle\preceq, b\rangle$ is forced to be stronger than $p$, first note that $\preceq\restriction \supp(p)=\preceq_p$ holds. To see that, let $x, y\in\supp(p)$ be any points. If $y\preceq_px$ then $a(x)=[x, \infty)\subseteq a(y)=[y, \infty)$. If $y<x$ are points in $\supp(p)$ then $x\not\preceq y$ by item (2) above. Finally, note that if $x$ is the smallest element of $\supp(p)$ such that $y\preceq x$, then $x\preceq y$. To see that, note that the assumptions imply that $a(y)=[x,\infty)$. The definition of $\preceq$ implies that $x\preceq y$.
\end{proof}

For future reference, note that the set $\supp(p)$ is dense in the preorder $\preceq$ by the definitions. Also, for each $y\in X$, let $B_y=\bigcap\{O\colon \langle y, O\rangle\in b\}$ and observe that $B_y=\{x\in X\colon x\preceq y\}$. This feature, that the initial segments of $\preceq$ which include their supremum are the largest possible $G_\gd$ sets, is the whole purpose of the construction of the preorder $\preceq$.

Now, I must show that $\langle \coll(\gw, X), \langle \preceq, \check b\rangle\rangle$ is a balanced pair. Suppose that $V[G_0], V[G_1]$ are mutually generic extensions and $p_0=\langle\preceq_0, b_0\rangle$ and $p_1=\langle\preceq_1, b_1\rangle$ be any conditions in these respective models stronger than $\langle\preceq, b\rangle$. I must find a common lower bound of conditions $p_0, p_1$. The common lower bound $p_{01}=\langle\preceq_{01}, b_{01}\rangle$ will have $\supp(p_{01})=\supp(p_0)\cup\supp(p_1)$ and $b_{01}=b_0\cup b_1$. The definition of the preorder is the more difficult point.

I will have $\preceq_0\cup\preceq_1\subset\preceq_{01}$, and, naturally, if $x\preceq_0 y\preceq_1 z$ for $x\in\supp(p_0)$, $y\in V$ and $z\in\supp(p_1)$, then $x\preceq_{01}yz$. Finally suppose that $X\cap V=c\cup d$ is a cut in the preorder $\preceq$ and containing some elements of $\supp(p_0)$ (gathered in the set denoted by $e_0$) and some lements of $\supp(p_1)$ (gathered in a set $e_1$). Note that this means that $d$ has no smallest element. I must show how to linearly preorder $e_0\cup e_1$ in $\preceq_{01}$.

Let $f_0, f_1$ be the respective longest prewellordered initial segments of $\langle e_0, \preceq_0\rangle$ and $\langle e_1, \preceq_1\rangle$ inside the cut $c\cup d$. For definiteness assume that $f_1$ is not shorter than $f_0$. Then, for $x_0\in e_0$ and $x_1\in e_1$, put $\langle x_0, x_1\rangle$ into $\preceq_{01}$ if $x_0\in f_0$ and $x_1$ is either in $e_1\setminus f_1$ or $x_1$ is in $f_1$ at an ordinal rank at least that of the rank of $x_0$ in $f_0$. Otherwise, put $\langle x_1, x_0\rangle$ into $\preceq_{01}$.

\begin{claim}
$p_{01}\in P$.
\end{claim}

\begin{proof}
Suppose that $x_0, x_1\in\supp(p_{01})$, $x_0\preceq_{01} x_1$ and $\langle x_1, O\rangle\in b_{01}$ is a pair. I must show that $x_0\in O$ holds. For definiteness assume that $x_0\in\supp(p_0)$ and $x_1\in\supp(p_1)$ and $\langle x_1, O\rangle\in b_1$ all hold.

\textbf{Case 1.} There is $y\in V$ such that $x_0\preceq_0y_0\preceq_1 x_1$. In such a case, in the ground model consider the set $B_y$. Since $p_0\in P$, it must be the case that $x_0\in B_y$. Since $p_1\in P$, it must be the case that $B_y\cap V\subset O$, and by a mutual genericity argument $B_y\cap V[G_0]\subset O$. It follows that $x_0\in O$ holds as desired.

\textbf{Case 2.} Not case 1. Then $x_0, x_1$ are in the same $\preceq$-cut $X\cap V=c\cup d$. Since $c,d\in V[G_0]\cap V[G_1]$, a mutual genericity argument shows that $c, d\in V$. Since $p_0, p_1\leq p$, it must be the case that the set $d$ has no smallest element. Work in $V$. Note that the set $\supp(p)$ is coinitial in $d$ and that $c=\bigcap\{ B_y\colon y\in\supp(p)\cap d\}$ is a $G_\gd$-set. Since $p_0\in P$, it must be the case that $x_0\in c$. Since $p_1\in P$, it must be the case that $c\cap V\subset O$, and by a mutual genericity argument $c\cap V[G_0]\subset O$. It follows that $x_0\in O$ holds as desired.
\end{proof}

\noindent It is now not difficult to show that $p_{01}$ is a common lower bound of $p_0$ and $p_1$. To show for example that $p_{01}\leq p_0$, suppose that $x\in\supp(p_{01})$ and suppose that $y$ is the smallest element of $\supp(p_0)$ such that $x\preceq_{01} y$; it must be proved that $y\preceq_{01}x$ holds. If $x\in\supp(p_0)$ then this is clear from the fact that $\preceq_0$ is a preorder. Suppose then that $x\in\supp(p_1)$ holds. If $y$ is $\preceq_0$-equivalent to some element in $V$, then $y\preceq_{01}x$ follows from $p_1\leq \langle \preceq, b\rangle$. Suppose then that $y$ is not $\preceq_0$-equivalent to any point in $V$. Let $X\cap V=c\cup d$ be the cut in $\preceq$ into which both $x, y$ fall. It cannot be the case that $y$ falls into the illfounded part of $\supp(p_0)$ inside this cut, since that illfounded part is convex in $\preceq_{01}$ and does not have a smallest element. Thus, $y$ falls into the longest pre-well-ordered initial segment of this cut in $\supp(p_0)$. Since $x\preceq_{01}y$ holds, it must be the case that $x$ falls into the pre-well-ordered segment of this cut in $\supp(p_1)$ and the ordinal rank of $x$ and $y$ in these two must be the same. Again, $y\preceq_{01}x$ holds as desired.
\end{proof}

\bibliographystyle{plain}
\bibliography{odkazy,zapletal}

\begin{thebibliography}{10}

\bibitem{bartoszynski:set}
Tomek Bartoszynski and Haim Judah.
\newblock {\em Set Theory. On the structure of the real line}.
\newblock A K Peters, Wellesley, MA, 1995.

\bibitem{hilbert:basis}
David Hilbert.
\newblock Ueber die theorie der algebraischen formen.
\newblock {\em Math. Ann.}, 36:473--534, 1890.

\bibitem{jech:set}
Thomas Jech.
\newblock {\em Set Theory}.
\newblock Academic Press, San Diego, 1978.

\bibitem{kanovei:book}
Vladimir Kanovei.
\newblock {\em Borel Equivalence Relations}.
\newblock University Lecture Series 44. American Mathematical Society,
  Providence, RI, 2008.

\bibitem{z:geometric}
Paul Larson and Jind{\v r}ich Zapletal.
\newblock {\em Geometric set theory}.
\newblock AMS Surveys and Monographs. American Mathematical Society,
  Providence, 2020.

\bibitem{marker:book}
David Marker.
\newblock {\em Model theory: An introduction}.
\newblock Graduate Texts in Mathematics 217. Springer Verlag, 2002.

\bibitem{z:distance}
Jindrich Zapletal.
\newblock Coloring the distance graphs.
\newblock 2021.
\newblock arXiv:2201.00275.

\bibitem{z:ngraphs}
Jindrich Zapletal.
\newblock Coloring closed noetherian graphs.
\newblock 2022.
\newblock in preparation.

\bibitem{z:redundant}
Jindrich Zapletal.
\newblock Coloring redundant hypergraphs.
\newblock 2022.
\newblock in preparation.

\bibitem{zhou:isosceles}
Yuxin Zhou.
\newblock Coloring isosceles triangles in two dimensions.
\newblock submitted, 2022.

\end{thebibliography}

\end{document}